\documentclass[11pt,a4paper]{article}
\usepackage{a4wide}
\usepackage{makeidx}
\usepackage{amsmath,amssymb,amsbsy,enumerate,verbatim}
\usepackage{amsthm}
\usepackage{epsfig,fancyhdr,color,hyperref}%,showkeys,amsmidx

\usepackage[utf8]{inputenc}
\usepackage[T1]{fontenc}
  
 \makeindex   
 
 \renewcommand{\r}{\mathbb{R}}
 \newcommand{\C}{\mathbb{C}}

\newcommand{\bb}{\pmb{\beta}}

\newtheorem{theorem}{\rm\bf Theorem}[section]
\newtheorem{proposition}[theorem]{\rm\bf Proposition}

\theoremstyle{definition}

\theoremstyle{remark}

\def\interieur#1{\mathord{\mathop{\kern 0pt #1}\limits^\circ}}

\date{\today}

\title{Riemannian Surfaces with Simple Singularities}
\date{\today}
\author{Marc Troyanov}

\begin{document}

\maketitle
 
\begin{abstract} In this note we discuss the geometry of Riemannian surfaces having a discrete set of singular points. We assume the conformal structure
extends through the singularities and the curvature is integrable. Such points are called \emph{simple singularities}. We first describe them locally and
then globally using the notion of (real) divisor.  We formulate a Gauss-Bonnet formula and relate it to some asymptotic isoperimetric ratio. We prove a 
classifications theorem for flat metrics with simple singularities on a compact surface and discuss the  Berger--Nirenberg Problem on surfaces with a divisor. 
We finally discuss the relation with spherical polyhedra.

\medskip 

This article is a translation of the paper \cite{Troyanov1990}, to be included in the forthcoming  book
\emph{Reshetnyak's Theory of Subharmonic Metrics}  edited by Fran\c{c}ois Fillastre and Dmitriy Slytskiy and to
be published by Springer and the  Centre de recherches mathématiques (CRM) in Montr\'eal. 

\medskip 


\medskip 

\end{abstract}

\section{Local Description} 

Among the singular Riemannian metrics on surfaces, the simplest ones are those with isolated singularities. Away from a discrete set, these metrics are then  smooth (say of class $C^2$). We will make two additional  hypotheses which are natural from a geometric point of view.
The first concerns the conformal structure.  If $g$ is a Riemannian metric on a surface $S$ having an isolated singularity at $p$, and if $U$ is a neighborhood of $p$ homeomorphic to the disk, then $U' = U \setminus \{0\}$  has a well-defined conformal structure (since the metric $g$ is smooth on $U'$,  one may  apply Korn-Lichtenstein's Theorem). From  the classification of conformal structures on the annulus, we know that $U'$  is conformally equivalent to a standard annulus 
$$
  A_{\rho} = \{z \in \C \mid \rho < |z| < 1\},
$$
for some fixed parameter $\rho \in [0, 1)$.
In this paper, we will always assume  $U'$  to be conformally equivalent to the punctured disk $A_0$. In other words, we are assuming that the conformal structure of $U'$ extends to $U$ (i.e. the point $p$ is a removable singularity from the conformal viewpoint). 

In particular, every point of $S$ (singular or not) has a coordinate neighborhood  in which the metric  $g$  can be written as 
$$
  g = \rho(x,y)(dx^2 + dy^2) = \rho(z) |dz|^2,
$$
where $z = x+iy$ and $\rho$ is a positive function that is of class $C^2$ outside the singularities. Such coordinates are called \textit{isothermal coordinates}.

\smallskip

{\small As a counterexample, one may consider a non-negative function $\phi$ on $\C$ that vanishes exactly on a contractible compact subset $Q \subset \C$. Let us denote by $S = \C/Q$ the space obtained by identifying all the points of $Q$ and $g = \phi(z) |dz|^2$. If $Q$ contains more than one point, then $(S,g)$ has a singular point $q = [Q]$ that does not satisfy the above condition.
}

\medskip

Our  second assumption concerns the curvature. It says that if $K$ denotes the curvature   and $dA$ the area element of $g$, then
$$
 \int_{U'} |K| dA < \infty.
$$

An important class of singularities satisfying the above conditions is given by the simple singularities:

\medskip

\textbf{Definition.}    \index{Simple singularity}
A conformal metric $g$ on a Riemann surface $S$ is said to have a \textit{simple singularity 
of order} $\beta$ at $p\in S$ if it can be locally written as 
$$
  g = e^{2u(z)} |z|^{2\beta} |dz|^2,
$$
where $\beta$ is a real number and $u$ is a function satisfying
$$
  u  \in L^1  \quad  \text{and}  \quad   \Delta u  \in L^1.
$$
In this definition, $z = x + iy$ is a local coordinate on $S$ defined in a neighborhood $U$ of $p$ and such that $z (p) = 0$. The Lebesgue space $L^1$  is defined  with respect to the  Lebesgue measure $dx dy$ on $U$ and the Laplacian of $u$  is defined in the sense of distributions by
$\Delta u= - \frac{\partial^2}{\partial x^2}   - \frac{\partial^2}{\partial y^2}$.

\medskip

Simple singularities naturally appear  in several  contexts as we shall soon see. A first class of examples 
is given by the following result due to MacOwen \cite[Appendix B]{McO}:

\begin{theorem}
 Let $g = e^{2u} |dz|^2$  be a conformal metric on the unit disk $D = \{z \in \mathbb{C} \mid   |z| <1\}$ having a singularity at the origin. Suppose that $g$ is smooth on the punctured disk  $D' =  D \setminus \{0\}$.
If there exist $\ell \in \r$ and $a, b> 0$ such that the curvature $K$ of $g$ satisfies
$$
 - b|z|^{\ell} \leq K(z) \leq  - a|z|^{\ell}, 
$$
then $0$ is a simple singularity of $g$.
\end{theorem}

A simple singularity of order $\beta < -1$  is always at infinite distance while a simple singularity of order $\beta > -1$ is always at a finite distance of ordinary points. For a singularity of order $- $1, both cases can occur, see \S 2.2 in  \cite{HT}. A \textit{cusp} \index{Cusp} is a simple singularity  of order  $\beta = -1$ admitting a neighborhood of finite area.

\medskip

A simple singularity of order  $\beta > -1$  is also called a \textit{conical singularity}   \index{Conical Singularity}  
of (total)  angle $\theta = 2\pi  (\beta+ 1)$. Such a singularity can indeed be approximated by a Euclidean cone of total angle $\theta$. In particular, if the curvature of $g$ is  bounded in some neighborhood of a conical singularity, there exists an ``exponential map'' making it possible to parametrize a neighborhood of the conical singularity by a neighborhood of the vertex of its tangent cone. In other words, one can introduce polar coordinates near a conical singularity. Moreover, if the curvature is continuous, then these polar coordinates are of class $C^1$ with respect to the isothermal coordinates, see \cite{Troyanov1990b} for proofs of these facts. 

\section{Global   description}

To investigate singular surfaces  having several simple singularities, it is convenient to introduce the following notion:
\medskip 

\textbf{Definition.}  \index{Divisor}
A   \textit{divisor}  on a Riemann surface $S$ is a formal sum 
$\pmb{\beta} = \sum_{i=1}^n \beta_i p_i$.
The \textit{support} of this divisor is the set $\mathrm{supp} (\bb) = \{p_1, \dots, p_n\}$. A conformal metric $g$ on $S$ \textit{represents the divisor} $\bb$ if it is smooth on the complement of $\mathrm{supp} (\bb)$ and if $g$ has a simple singularity of order $\beta_i$ at $p_i$  for $i = 1, \dots, n$.

\medskip 

\textbf{Examples.} \  (1) The metric $g = |dz|^2$ on the Riemann sphere $\C \cup \{\infty\}$ represents the divisor $\bb = (-2)\cdot \infty$. 

\medskip

(2) More generaly, the metric $g = |z|^{2\alpha}|dz|^2$ on $\C \cup \{\infty\}$ represents the divisor $\bb = \alpha \cdot 0 +  (-2- \alpha)\cdot \infty$. 

\medskip

 (3) If $\omega = \varphi (z) dz$ is a meromorphic differential on the Riemann surface $S$, then $g = |\omega|^2$ is a flat Riemannian metric with simple singularities representing the divisor $\bb = \mathrm{div} (\omega)$. 

\medskip

(4)  If $(S_1,g_1)$ is a smooth Riemannian surface and $f : S \to S_1$ is a branched  covering, then $g = f^*(g_1)$ is a Riemannian metric on $S$ representing the ramification divisor of $f$, that is $\bb = \sum_p O_p(f) \cdot p$, where $O_p(f)$ is the ramification order of $f$ at $p$  (i.e. the local degree minus $1$).

\medskip

(5) If $S$ is a two-dimensional polyhedron (euclidean, spherical or hyperbolic) with vertices $p_1, \dots, p_n$, then the metric induced by the geometric realization of that polyhedron represents the divisor $\pmb{\beta} = \sum_{i=1}^n \beta_i p_i$, where 
$2\pi(\beta_i+1) = \theta_i$ is the sum of the angles at $p_i$ of all faces incident with $p_i$.

\medskip

(6) Let $(\tilde{S},\tilde{g})$ be a smooth Riemannian surface on which a finite group $\Gamma$ acts by isometries. If
$S = \tilde{S}/\Gamma$ is a surface without boundary, then it inherits a Riemannian metric with simple singularities representing the divisor
$\pmb{\beta} = \sum_{i=1}^n \beta_i p_i$, where $\beta_i = (\frac{1}{n_i}-1)$. Here, the point $p_i$ is the image of a point
$\tilde{p}_i \in \tilde{S}$ such that $\Gamma$ has a stabiliser of order $n_i$ at $\tilde{p}_i $. This examples generalizes to two-dimensional orbifolds.

\medskip 

In examples 4 to 6, all the singular points are conical singularities. An important source of examples, where no singular point is conical, is given by the following Theorem:
\begin{theorem}\label{th.Huber} 
 Let $(S', g ')$ be a complete Riemannian surface of class $C^2$ with finite total curvature: $\int_{S'}|K|dA < \infty$. Then there exists a compact Riemann surface $S$, a divisor $\pmb{\beta} = \sum_{i=1}^n \beta_i p_i$ on $S$  such that $\beta_i \leq -1$ for all $i$,  and a conformal metric $g$ on $S$ representing this divisor such that   $(S', g ')$ is isometric to
$(S \setminus \mathrm{supp}(\bb), g)$.
\end{theorem}
This result is essentially due to A. Huber, we  refer to \S 1.1 and  \S 2.9 in  \cite{HT} for a discussion and a proof  of Huber's theorem in the above formulation.  
It is not difficult to see that if the surface  $(S', g ')$  has finite area, then $\beta_i = -1$ for all $i$.

\section{Some global geometry}

For compact Riemannian surfaces with simple singularities, there is a well known Gauss-Bonnet Formula (see e.g. \cite{Finn} for the case where all orders satisfy $\beta_i \leq -1$). To state the Formula we define  the \textit{Euler characteristic of a surface $S$ with divisor} as  $\pmb{\beta} = \sum_{i=1}^n \beta_i p_i$ as $\chi(S, \bb) = \chi(S) + \sum_i \beta_i$.

\begin{theorem}[The Gauss-Bonnet Formula]  \index{Gauss-Bonnet Formula} \label{GB}
Let $(S, g)$ be a compact Riemannian surface whose metric represents a divisor $\bb$. Then the total curvature of $(S, g)$ is finite and 
we have
$$
 \frac{1}{2\pi}\int_S K dA = \chi(S, \bb).
$$
\end{theorem}

See e.g.  \cite[Theorem 2.8]{HT} for a proof.

\medskip 

For example, if $\omega$ is a meromorphic differential on the closed Riemann surface of genus $\gamma$, then the Gauss-Bonnet Formula implies that the degree of $\omega$ (i.e. the  number of zeroes minus the number of poles), is equal to $2\gamma  - 2$. Indeed, $g = |\omega|^2$ is a flat metric representing the divisor $\mathrm{div}(\omega)$. Another application of the Gauss-Bonnet formula is the Riemann-Hurwitz formula:

\begin{proposition}[The Riemann-Hurwitz Formula]   \index{Riemann-Hurwitz Formula}
Let $f : S \to S_1$ be a branched cover of degree $d$ between two closed surfaces, then 
$$
  \chi(S) + \sum_{p\in S} O_p(f)  = d \chi(S_1),
$$
where $O_p(f)$ is the branching order of $f$ at $p$. 
\end{proposition} 

\begin{proof}
Pick an arbitrary smooth metric $g_1$ on $S_1$ and set $g = f^*(g_1)$. Then $g$ is a Riemannian metric with simple singularities on
  $S$ representing the ramification  divisor of $f$. The above formula follows now from Theorem \ref{GB},  since we obviously have 
$$
\int_S K dA = d \cdot \int_{S_1} K_1 dA_1. 
$$
\end{proof}

\medskip

Huber's theorem, together  with the Gauss-Bonnet Formula, is a refinement of the Cohn-Vossen inequality:

\begin{proposition}[The Cohn-Vossen inequality]  \index{Cohn-Vossen inequality}
Let $(S', g')$ be a compact Riemannian surface of class $C^2$ with finite total curvature: $\int_{S'} |K'| dA' < \infty$, then we have
$$
 \frac{1}{2\pi}\int_{S'} K' dA' \leq  \chi(S').
$$
Moreover  we have equality if  $(S', g')$ has finite area.
\end{proposition}

\begin{proof}
Huber's Theorem tells us that $(S', g')$ admits a compactification $(S, g)$ where $g$ is a metric representing a divisor
$\pmb{\beta} = \sum_{i=1}^n \beta_i p_i$ such that $\beta_i \leq -1$ for all $i$. We then have from the Gauss-Bonnet Formula
$$
   \frac{1}{2\pi}\int_{S'} K' dA' =  \frac{1}{2\pi}\int_{S} K dA = \chi(S) +  \sum_{i=1}^n \beta_i  \leq \chi(S) -n = \chi(S').
$$
If $S'$ has finite area, then we have  $\beta_i =-1$ for all $i$ and the above inequality is in fact an equality. 
\end{proof}
{\small Note that although the Cohn-Vossen inequality is  a consequence of Huber's Theorem  \ref{th.Huber} and the Gauss-Bonnet Formula, 
one should not  consider it to be a corollary of these results. The reason is that the proof of Huber's theorem is in part based on the Cohn-Vossen inequality.}

\medskip

The difference between $\chi(S')$ and $\chi(S, \bb)$ in the Cohn-Vossen inequality is an isoperimetric constant: 

\begin{theorem}
Let $(S, g)$ be a compact Riemannian surface whose metric $g$ represents a divisor $\pmb{\beta} = \sum_{i=1}^n \beta_i p_i$ such that $\beta_i \leq -1$ for all $i$.  
Fix a point $q$ on $S' = S\setminus \{p_1, \dots, p_n\}$ and denote by  $A(q, r)$  the area of $B_q(r) := \{x \in S' \mid d(q,x) \leq r\}$ and $L(q, r)$ the length of $\partial B_q(r)$. Then
$$
 \lim_{r \to \infty} \frac{L^2(q,r)}{4\pi A(q,r)}  = - \sum_{i=1}^n (\beta_i +1) = \chi(S') -  \chi(S, \bb).
$$
\end{theorem}

This result is due to K. Shiohama \cite{Shiohama}, but R. Finn obtained a partial result in this direction \cite[Theorem 10]{Finn}.

\section{Classifying Flat Metrics}

Let us now formulate a classification theorem for flat metrics with simple singularities on a compact surface.

\begin{theorem}
 Let $S$ be a compact Riemann surface with divisor $\bb = \sum \beta_ip_i$. Then there exists a conformal flat metric representing $\bb$ on $S$ if and only if $\chi(S,\bb) = 0$. Moreover this metric is unique up to homothety.
\end{theorem}

This theorem has several proofs, see e.g. \cite{Troyanov1986} and \cite[\S 7]{HT}. We  give  here the proof  from \cite{HT}.

\begin{proof}
Introduce in the neighborhood  of each $p_i$  a coordinate $z_i$  such that $z_i(p_i) = 0$ and choose an arbitrary conformal metric $g_0$ on $S$ such that  $g_0 = |dz_i|^2$ in the neighborhood of  each $p_i$. Let us now choose a positive function $\rho : S \to \r$ which is of class $C^2$ on  $S\setminus \mathrm{supp}(\bb)$  and such that $\rho = |z_i|^{2\beta_i}$ in the neighborhood of each $p_i$.
The metric $g_1 = \rho g_0$ is then a conformal metric representing the divisor $\bb$. 

Since the desired metric must be conformal on $S$, it can be written as $g = e^{2u} g_1$. Note that if $u$ is a function of class $C^2$ on $S$ such that 
\begin{equation}\label{LaplK1}
  \Delta_1 u = -K_1,
\end{equation}
where $\Delta_1$ and $K_1$ denote   the Laplacian and the curvature of $g_1$, then  $g = e^{2u} g_1$  is a flat conformal metric representing $\bb$ on $S$. 
Because $\Delta_1$ is a singular operator,  it is more convenient  to write the previous equation as
\begin{equation}\label{LaplK2}
  \Delta_0 u = -\rho K_1,
\end{equation}
where $\Delta_0$ is the Laplacian of the smooth metric  $g_0$.

Note that \eqref{LaplK1} and  \eqref{LaplK2} are equivalent equations, but since $K_1$ vanishes in a  neighborhood of the points $p_i$ and the functions $\rho$ and $K_1$  are of class $C^2$ on  $S\setminus \mathrm{supp}(\bb)$, the right hand side  of \eqref{LaplK2} is  of class $C^2$ on the whole surface $S$.

It is well known that the partial differential equation \eqref{LaplK2} has a solution if and only if the integral of the right hand side vanishes, which  follows from the Gauss-Bonnet formula:
$$
 \int_S K_1 \rho dA_0 = \int_S K_1 dA_1 = 2\pi \chi(S,\bb) = 0.
$$
We have thus proved the existence of a flat conformal metric on $S$ representing the divisor $\bb$. 
The uniqueness follows from the fact that if $g_1$ and $g_2$ are two such metrics, then $g_2 = e^{2v} g_1$ for a harmonic function $v$ on $S$. This function is constant (because $S$ is a closed surface) and the two metrics are therefore homothetic.
\end{proof}
 
\medskip

The above Theorem gives us a short proof of the Uniformization Theorem for the sphere:

\begin{theorem}  \index{Uniformization Theorem (for the 2-sphere)}
  Any  Riemann surface homeomorphic to the two-sphere  is conformally equivalent to $\mathbb{C} \cup \{\infty\}$.
\end{theorem}

\begin{proof}
Let us choose a point $p$ in $S$ and consider the divisor $\bb = (-2)\cdot p$. Observe that $\chi(S, \bb) = 2 - 2 = 0$.
The previous theorem tells us   that there is a conformal  flat metric $g$ on $S$ representing this divisor. It is clear that $(S, g)$ is isometric (and thus conformally equivalent) to   $\left(\mathbb{C} \cup \{\infty\}, |dz|^2\right) $.
\end{proof}
 
\medskip

{\small  Of course the Uniformization  Theorem also  means   that  there exists a smooth  conformal  metric of constant positive curvature on $S$. However, it is hard to prove this result by directly solving the corresponding Berger-Nirenberg problem, that is by  directly constructing a conformal metric of curvature $+1$. The above proof on the other hand is almost trivial.}

\section{The  Berger--Nirenberg Problem  on Surfaces with Divisors} \index{The  Berger--Nirenberg Problem}

The classical  Berger-Nirenberg problem is the following:

\medskip 

\textbf{Problem 1.} Let $S$ be a Riemann surface and  $K : S \to  \mathbb{R}$ a function on this surface. Is there a conformal metric on $S$ whose curvature is the function $K$? If it exists, is such a metric unique?

\medskip 

This problem is clearly not well posed  for open surfaces. One could hope that the   problem is well posed for  complete Riemannian metrics, however, it is not difficult to construct families $\{g_{\lambda}\}$ of conformal metrics on a Riemann surface which are complete, conformal, of the same curvature and whose geometry at infinity drastically varies with $\lambda$,  in the sense that they are not mutually bilipschitz.   An example is given in \cite{HT1990}.
The previous discussion, in particular Huber's Theorem \ref{th.Huber}, suggests to replace the Berger-Nirenberg Problem on open  surfaces by a version of the  problem on compact surfaces with a divisor.

\medskip 

\textbf{Problem 2.} Let $(S,\beta)$ be a compact  Riemann surface  with divisor, and  $K : S \to  \mathbb{R}$ be a smooth function. Is there a conformal metric $g$ on $S$ that  represents $\bb$ and whose curvature is the function $K$? If it exists, is such a metric unique?

\medskip 

We have already answered this question when $K$ vanishes everywhere.

\smallskip 

Problem 2 is studied in the papers \cite{Troyanov1991} (in the case of conic singularity) and \cite{HT} in the general case. The results can be summarized in a form which is similar to the classical theory in the smooth case as it is exposed   in the foundational article  \cite{KW} by Jerry Kazdan and Frank Warner.

\begin{theorem}\label{prescrK}
Let $(S, \bb)$ be a compact Riemann surface with a divisor $\bb = \sum \beta_ip_i$ , and $K : S \to \r$ be  a smooth function. Suppose that there exists $p>1$ such that   $h_i(z) = |z-p_i|^{2\beta_i}K(z)$ is a function of class $L^p$ in a neighborhood of each $p_i$. Moreover
\begin{enumerate}[(a)]
\item If $\chi(S,\bb) >0$, we assume  $\sup(K) >0$ and $q\chi(S,\bb)<2$, where $q=p/(p-1)$.
\item If $\chi(S,\bb) =0$, we assume either that $K \equiv 0$  or  $\sup(K) >0$  and $\int_S KdA_0 < 0$, where $dA_0$
is the area element of a flat conformal metric representing $\bb$.
\item If $\chi(S,\bb) <0$, we assume $K \leq 0$ and $K \not\equiv 0$.
\end{enumerate}
Then there exists a conformal metric $g$ on $S $ which represents the divisor $\bb$
and whose curvature is $K$. In case (c), this metric is unique.
\end{theorem}

A very brief idea of the proof is presented in \cite{HT1990} (see \cite{Troyanov1991} and \cite{HT} for details).
Some particular cases of this theorem have been obtained previously by W. M. Ni, R. MacOwen and P. Aviles. At the beginning of the twentieth century, Emile Picard had already studied the case of curvature $- 1$ in \cite{Picard}.
The hypotheses of the previous theorem impose a decay of the curvature  when approaching the singularities of order $< - 1$. The next result, only valid for non-positive curvature,  does not impose such a behavior.

\begin{theorem}
Let $S$ be a compact Riemann surface and $g_1$ be a conformal metric representing a divisor $\bb = \sum_{i=1}^n  \beta_ip_i$
such that $n \geq 1$ and $\chi(S,\bb) < 0$. Let  $K : S \to \r$  be a smooth nonpositive function  such that 
$$ bK\leq K_1\leq aK<0$$
on the complement of  a compact subset of $S' = S \setminus \{p_1, \dots, p_n\}$, where $K_1$  is the curvature of $g_1$ and $a,b$  are positive constants. Then there exists a unique conformal metric $g$ on $S$ which represents $\bb$, has curvature $K$ and is 
conformally quasi-isometric to $g_1$.
\end{theorem}

See \cite[Theorem 8.1, 8.4]{HT}  or  \cite{McO} for the proof.  As an application, using this  Theorem,  one can  construct metrics with prescribed (negative) curvature having cusps. The previous Theorem  also admits a generalization to non-compact Riemann surfaces of finite type having hyperbolic ends.

\medskip

We end this section  with two  results on the Berger-Nirenberg Problem on closed Riemann surfaces with divisor. The first result, proved in \cite{Tang} by Junjie Tang,  allows us to solve the Berger-Nirenberg Problem  when $\chi (S, \bb)$ is small enough.

\begin{theorem}
Let $S$ be a closed Riemann surface with a divisor  ${\pmb{{\pmb{\alpha}}}} = \sum_{i=1}^n \alpha_i p_i$ such that $\chi(S, {\pmb{\alpha}}) = 0$, and let
us pick a conformal metric $g_0$ representing ${\pmb{\alpha}}$. 

Suppose another divisor $\bb= \sum_{i=1}^n \beta_i p_i$ with same support is given on $S$,  such that  $\chi(S, \bb) < 0$ and consider a  function $K : S\to \r$  such that  $K = O(|z - p_i|^{\ell_i})$ (in the neighborhood of each $p_i$), where $\ell_i > -2(1+\alpha_i)$.

\smallskip 

\emph{(A)}  If $\beta_i \leq \alpha_i$ for all $i$, then a necessary condition for the existence of a conformal metric $g$ on $S$ representing 
$\bb$ is
\begin{equation}\label{kaneg}
  \int_S KdA_0 < 0,
\end{equation}
where $dA_0$ is the area element of $g_0$.

\smallskip 

\emph{(B)}  There exists $\varepsilon >0$  (depending on $S$, $\alpha$ and $K$) such that if $\max_i |\beta_i - \alpha_i| \leq \varepsilon$,
then \eqref{kaneg} is a sufficient condition for the existence of a conformal metric $g$ representing $\beta$.
\end{theorem}

\medskip

The second result is due to Dominique Hulin \cite{Hulin1993}. It  allows us to solve the problem when $K$ is close enough  to a given non-positive function: 
\begin{theorem} 
Let $S$ be a compact Riemann surface with a divisor $\bb= \sum_{i=1}^n \beta_i p_i$  such that $\chi(S, \bb) < 0$.
Let $K_1, k : S \to \mathbb{R}$ be two smooth functions on $S$ such that $K_1 \leq 0 \leq k$ everywhere on $S$, and $K_1 \not\equiv 0$.
Suppose that $|z-p_i|^{2\beta_i}(k(z) - K_1(z)) \in L^p$ for some $p>1$ in the neighborhood of the points  $p_i$. 
\  Then there exists a constant $C > 0$ (depending on $S$, $\bb$, $K_1$ and $k$) such that if
$$
K_1\leq K \leq K_1 + Ck
$$
on $S$, then there exists a conformal metric $g$ of curvature $K$ on $S$ which represents $\bb$. 
\end{theorem}
The dependence of the constant $C$ on $S$, $\bb$, $k$ and $K_1$ is explicitly  given in  \cite[Theorem 6.1]{Hulin1993}

\newpage

\section{Spherical Polyhedra}

The following result classifies spherical surfaces with less than three conical points.
\begin{theorem}
Let $g$ be a metric on the sphere $S^2$ representing a divisor $\bb = \beta_1p_1 + \beta_2p_2$   and whose curvature is constant $K = + 1$. Then $\beta_1 = \beta_2$,  moreover:
\begin{enumerate}[(a)]
\item  If $\beta_i$ is not an integer, then $p_1$ and $p_2$ are conjugate points (that is $d(p_1,p_2) = \pi$).
\item  If $\beta = m \in \mathbb{N}$, then   $(S, g)$ is isometric to a branched covering of degree $m+1$ of the standard sphere  $\mathbb{S}^2$, 
(with its canonical metric) branched over two points, with  ramification order equal to $m$.
Moreover, two such metrics are isometric if and only if their singularities are of the same order and separated by the same distance.
\end{enumerate}
\end{theorem}

Let us call  \emph{spherical polyhedra}    \index{Spherical polyhedra} a Riemannian surface homeomorphic to the sphere, with conical singularities of order 
$\beta_i  \in (-1,0)$ and whose  Gauss curvature is constant $K = +1$. 
A fundamental Theorem by A. D.  Aleksandrov  states that a Riemannian surface, homeomorphic to the sphere, with   constant curvature  $K = +1$ and
 conical singularities of order $\beta_i \in (-1,0)$  can be realized as the boundary  of a convex polytope 
in the standard three-sphere $\mathbb{S}^3$ (see  \cite[Chapter XII, p. 400]{Alexandrov} ).  Note that from the previous Theorem,  a spherical polyhedron cannot have exactly one singularity.

\medskip 

The last result  classifies the divisors that can be represented by a spherical polyhedron having at least three singularities:
\begin{theorem}
Let $\bb = \sum_{i=1}^n \beta_i p_i$ be a divisor on $S^2 = \mathbb{C} \cup \infty$ such  $n\geq 3$ and that $- 1 < \beta_i < 0$ for all $i$.
Then there exists a unique conformal metric $g$ with constant curvature $K = +1$ representing $\bb$ if and only if
\begin{equation}\label{CondLT}
  0 < 2 + \sum_{i=1}^n \beta_i < 2(1 + \min_i \{\beta_i\}).
\end{equation}
\end{theorem}
Expressed with the cone angles $\theta_i = 2\pi(1+\beta_i)$, the hypothesis can also be written as $\theta_i < 2\pi$ and 
\begin{equation}\label{CondLT2}
 0 < 4\pi +  \sum_{i=1}^n (\theta_i - 2\pi)  < 2  \min_i \{\theta_i\}.
\end{equation}
The first inequality in this condition is none other than the Gauss-Bonnet formula.

Let us observe that the condition \eqref{CondLT2} is similar to the condition satisfied by the angles 
$\varphi_1, \dots, \varphi_n$   of a spherical convex polygon:
$$
 0 < 2\pi +  \sum_{i=1}^n (\varphi_i - \pi)  < 2  \min_i \{\varphi_i\}.
$$
The existence of a spherical metric representing $\bb$ follows from Theorem \ref{prescrK} above (see also \cite{Troyanov1991}). The necessity of \eqref{CondLT}, as well as the uniqueness of the metric have been proved by Feng Luo and Gang Tian  in \cite{LT}. 

\bigskip

\textit{Note added in 2021.} 
The subject of this survey has grown considerably over the past 30 years. A nice reference covering more recent aspects of the theory is the article \cite{Lai} by M. Lai

%___________________  

 \printindex

\end{document}